\documentclass[a4paper,12pt]{article}%

\usepackage{amsmath}%
\usepackage{amsfonts}%
\usepackage{amssymb}%
\usepackage{graphicx}
\newtheorem{proposizione}{Proposition}

\newtheorem{definizione}{Definition}
\newenvironment{proof}[1][Proof]{\textbf{#1.} }{\ \rule{0.5em}{0.5em}}

\begin{document}

\title{A remark on  Domain Decomposition approaches  solving Three Dimensional Variational Data Assimilation models.}
\author{Luisa D'Amore, Rosalba Cacciapuoti \\Department of Mathematics and Applications, \\ University of Naples Federico II, Naples, ITALY \\ luisa.damore@unina.it, rosalb.cacciapuoti@studenti.unina.it}
\date{}
\maketitle
\begin{abstract}
\noindent Data Assimilation (DA) is a methodology for combining mathematical models
simulating complex systems (the background knowledge) and measurements (the reality or
observational data) in order to improve the estimate of the system state. This  is a large scale ill posed inverse problem then in this note we consider the Tikhonov-regularized variational formulation of 3D- DA problem, namely the so-called 3D-Var DA problem. We review two Domain Decomposition (DD) approches, namely the functional DD and the discrete Multiplicative Parallel Schwarz, and as the 3D-Var DA problem is a least square problem,  we  prove the equivalence between these approches.
\end{abstract}

\section{Introduction}
The DD methods are well established techniques for solving boundary-value problems \cite{DD1,DD2}. The earliest known DD method was proposed in the pioneering work of H. A. Schwarz in 1869 \cite{S}. Renewed interest in these methods was sparked by the advent of parallel
computing, and the Parallel Schwarz Method (MPS) was introduced by P.L. Lions in 1988 \cite{MPS}. In \cite{tesi} the MPS is applied for solving a three dimensional variational Data Assimilation (DA) problem, which is  a large scale inverse and ill posed problem  used to handle a huge amount of data and   requiring  new mathematical and algorithmic approaches for its solution \cite{DA,articoloAIP,articoloJNAIAM,articolo,articoloJCP,DA2}. In this note, we review the two DD approches, namely the  one  introduced in \cite{articolo} and the   MPS method applied to the Euler-Lagrange equations arising from the VarDA minimization problem. We prove equivalence between these approaches.\\ The note is organized as follows: in section 2 we briefly review  DA inverse problem and its  variational formulation \cite{T}; in section 3 we apply the DD approches and we prove the main result.

\section{The DA inverse problem}
Let $x \in \Omega \subset \mathbb{R}^{N}$, $t \in [0,T]$ and let $u(t,x)$ be the state evolution of a predictive system from time $t - \Delta t$ to time $t$, governed by the mathematical model $\mathcal{M}_{t}[ u(t,x)]$. So, it is

\begin{displaymath}
\mathcal{M}_{t-\Delta t,t}: \ \ u(t-\Delta t,x) \to u(t,x).
\end{displaymath}
Let $\{t_{k}\}_{k=0,1,...}$ be a discretization of interval of time $[0,T]$, where  $t_{k}=t_{0}+k\Delta t$, and let $D_{NP}(\Omega)={\{(x_{j})\}}_{j=1,...,NP} \in \mathbb{R}^{NP \times N}$, be a discretization of $\Omega \subset \mathbb{R}^{N}$, where $x_{j}\in \Omega$. \\
Let:
\begin{displaymath}
v(t,y)=\mathcal{H}(u(t,y)), \quad y \in \Omega
\end{displaymath} 
denote the observations mapping, where $\mathcal{H}$ is a given nonlinear operator.

For each $k = 0,1,...$, we consider
\begin{itemize}

\item $\textbf{u}_{k}^{b}=\{{u}_{k}^{j}\}_{j=1,...,NP}^{b} \equiv \{u(t_{k},x_{j})^{b}\}_{j=1,...,NP} \in \mathbb{R}^{NP}$: (background) numerical solution of the model $\mathcal{M}_{t}[ u(t,x)]$ on $\{t_{k}\} \times D_{NP}(\Omega)$; 

\item  $\textbf{v}_{k}=\{v(t_{k},y_{j})\}_{j=1,...,nobs}$: the vector values of the observations on $y_{j}\in \Omega$ at time $t_{k}$;

\item  $\mathcal{H}(x) \backsimeq \mathcal{H}(y)+ \textbf{H}(x-y)$: a linearization of $\mathcal{H}$, where \textbf{H}$\in \mathbb{R}^{NP \times nobs}$ is the matrix
obtained by the first order approximation of the Jacobian of $\mathcal{H}$ and $nobs \ll NP$;
\item \textbf{R} and \textbf{B} the covariance matrices of the errors on the observations and on the background,
respectively. These matrices are symmetric and positive definite.
\end{itemize}

\begin{definizione}\label{defDA}(The DA inverse problem). The DA inverse problem is to compute the vector
 $\textbf{u}_{k}^{DA}=\{u_{k}^{j}\}_{j=1,...,NP}^{DA}$ such that:
 \begin{equation}\label{SDA}
\textbf{v}_{k}=\textbf{H}[\textbf{u}_{k}^{DA}].
\end{equation}
\end{definizione}

\noindent Since $\textbf{H}$ is typically rank deficient and highly ill conditioned, DA is an ill posed inverse
problem \cite{articolo}. The objective is determine the solution in least squares sense.
\begin{definizione}
The solution in least squares sense for the problem in (\ref{SDA}) is a vector \textbf{u}$^{DA}$ such that:
\begin{equation}\label{ls}
\textbf{u}^{DA}=argmin_{\textbf{u}\in \mathbb{R}^{NP}}\textbf{J}(\textbf{u})=argmin_{u}\left\{||\textbf{H}\textbf{u}-\textbf{v}||_{\textbf{R}}^{2}\right\}.
\end{equation}
\end{definizione}
\noindent The problem in (\ref{ls}) ignores the background, so we consider the following Tikhonov-regularized formulation. In the following we let time $t_{k}$
be fixed, i.e. we consider the so-called 3D-Var DA problem \cite{articolo}, then for simplicity of notations,
we refer to \textbf{u}$_{k}^{b}$ and \textbf{u}$_{k}^{DA}$ omitting index $k$.

\begin{definizione}(The 3D-Var DA problem). 3D Variational DA problem is to compute the vector $\textbf{u}^{DA}$ such that
\begin{equation}\label{3Dvar}
\textbf{u}^{DA}=argmin_{\textbf{u}\in \mathbb{R}^{NP}}\textbf{J}(\textbf{u})=argmin_{u}\left\{||\textbf{H}\textbf{u}-\textbf{v}||_{\textbf{R}}^{2}+\lambda ||\textbf{u}-\textbf{u}^{b}||_{\textbf{B}}^{2}\right\}
\end{equation}
where $\lambda$ is the regularization parameter.
\end{definizione}
\noindent When the regularization parameter $\lambda$ appraches to zero the regularized problem tends to the DA (ill
posed) inverse problem, while the increase the regularization parameter has the effect of decreasing the uncertainty in the
background. In the following we let $\lambda =1$ as we do not address the impact of the regularization parameter.
\\
The 3D-Var operator is:
\begin{equation}\label{3Doper}
\textbf{J}(\textbf{u})\equiv \textbf{J}(\textbf{u}, \textbf{R}, \textbf{B}, D_{NP}(\Omega))=(\textbf{H}\textbf{u}-\textbf{v})^{T}\textbf{R}(\textbf{H}\textbf{u}-\textbf{v})+  (\textbf{u}-\textbf{u}^{b})^{T}\textbf{B}(\textbf{u}-\textbf{u}^{b}).
\end{equation}
The matrix $\textbf{H}$ is ill conditioned so we consider the preconditioner matrix $\textbf{V}$ such that $\textbf{B}=\textbf{V}\textbf{V}^{T}$. Let $\partial \textbf{u}^{DA}={\textbf{u}^{DA}}-{\textbf{u}^{b}}$ and $\textbf{w}=\textbf{V}^{T} \partial \textbf{u}^{DA}$, the operator $\textbf{J}$ in (\ref{3Doper}) can be rewritter as follows:
\begin{equation}\label{funzionaleJ}
\textbf{J}(\textbf{w})=\frac{1}{2}\textbf{w}^{T}\textbf{w}+\frac{1}{2} {(\textbf{H}\textbf{V}\textbf{w}-\textbf{d})}^{T}\textbf{R}^{-1}(\textbf{H}\textbf{V}\textbf{w}-\textbf{d}),
\end{equation}
where 
\begin{equation}\label{vettored}
\textbf{d}=[\textbf{v}-\textbf{H}(\textbf{u})].
\end{equation}
\section{The DD approaches applied to DA inverse problem}
Our propose in \cite{tesi} has been applying the DD approach in \cite{MPS}, i.e. MPS, for solving three dimensional variational DA problem. \\ 
The discrete MPS used in \cite{tesi} is composed of the following steps:
\begin{enumerate}
\item Decomposition of domain $\Omega$ into a sequence of sub domains $\Omega_{i}$ such that:
\begin{displaymath}
\Omega=\bigcup_{i=1}^{N}\Omega_{i}.
\end{displaymath}
\item Definition of interfaces of sub domains $\Omega_{i}$ as follows:
\begin{equation}\label{interfacce}
\Gamma_{ij}:=\partial \Omega_{i} \cap \Omega_{j} \ \ \ \ \textrm{for $i,j=1,...,J$}.
\end{equation}
\item Definition of restriction matrices $R_{i}$, $R_{ij}$ to sub domain $\Omega_{i}$ and interface $\Gamma_{ij}$, and extension matrices $R_{i}^{T}$, $R_{ij}$ to domain $\Omega$ for $i,j=1,...,J$ as follows:
\begin{equation}\label{RjDA}
R_{i}=\bordermatrix{  \footnotesize
& & &  & \textrm{ \footnotesize $s_{i-1}+1$} &\cdots & \textrm{ \footnotesize $s_{i-1}+r_{i}$} & & & \cr
& 0 & \cdots & 0& 0 & \cdots & 0 &0& \cdots  & 0  \cr
& \vdots & &\vdots & \vdots  &  & \vdots &\vdots & & \vdots  \cr
& 0 & \cdots & & 0 & \cdots & 0 &0 & \cdots &0 \cr
\textrm{ \footnotesize $s_{i-1}+1$}& 0 & \cdots & 0 & 1&  & & 0 &\cdots & 0\cr
\vdots & \vdots & & \vdots &  &\ddots & & &  & & \cr
\textrm{ \footnotesize $s_{i-1}+r_{i}$} & 0 &\cdots  &0 & 0 & &1 & 0 & \cdots & 0  \cr
& 0 & \cdots & 0& 0 & \cdots & 0 &0& \cdots  & 0  \cr
& \vdots & &\vdots & \vdots  &  & \vdots &\vdots & & \vdots  \cr
& 0 & \cdots & 0& 0 & \cdots & 0 &0 & \cdots &0 \cr
},
\end{equation}

\begin{equation}
R_{ij}=\bordermatrix{
& & &  & \textrm{ \footnotesize $\bar{s}_{i-1,i}+1$} &\cdots & \textrm{ \footnotesize $\bar{s}_{i-1,i}+r_{i}$} & & & \cr
& 0 & \cdots & 0& 0 & \cdots & 0 &0& \cdots  & 0  \cr
& \vdots & &\vdots & \vdots  &  & \vdots &\vdots & & \vdots  \cr
& 0 & \cdots & & 0 & \cdots & 0 &0 & \cdots &0 \cr
\textrm{ \footnotesize $\bar{s}_{i-1,i}+1$} & 0 & \cdots & 0 & 1&  & & 0 &\cdots & 0\cr
\vdots & \vdots & & \vdots &  &\ddots & & &  & & \cr
\textrm{ \footnotesize $\bar{s}_{i-1,i}+r_{i}$} & 0 &\cdots  &0 & 0 & &1 & 0 & \cdots & 0  \cr
& 0 & \cdots & 0& 0 & \cdots & 0 &0& \cdots  & 0  \cr
& \vdots & &\vdots & \vdots  &  & \vdots &\vdots & & \vdots  \cr
& 0 & \cdots & 0& 0 & \cdots & 0 &0 & \cdots &0 \cr
}
\end{equation}
where $s_{i,j}=r_{i}-C_{i,j}$, $\bar{s}_{i,j}=s_{i,j}+t_{ij}$, and $r_{i}$, $t_{i,j}$, $C_{i,j}$ points of sub domain $\Omega_{i}$, interfaces $\Gamma_{ij}$ and sub domain $\Omega_{ij}=\Omega_{i}\cap \Omega_{j}$, respectively.
\item  For $i=1,2,...,J$,  solution of $J$ subproblems $P_{i}^{n+1}$, for $n=0,1,2,...$ where 
\begin{equation}\label{argmin}
\begin{split}
P_{i}^{n+1} \ \ argmin_{\textbf{u}_{i}^{n+1}\in \mathbb{R}^{r_{i}}}\textbf{J}_{i}(\textbf{u}_{i}^{n+1}),
\end{split}
\end{equation}
where
\begin{equation}\label{funzionaliJi}
\textbf{J}_{i}(\textbf{u}_{i}^{n+1})=||\textbf{H}_{i}\textbf{u}_{i}^{n+1}-\textbf{v}_{i}||_{\textbf{R}_{i}}^{2}+\\ ||\textbf{u}_{i}^{n+1}-{({\textbf{u}_{i}}^{b})}||_{\textbf{B$_{i}$}}^{2}+ ||\textbf{u}_{i}^{n+1}/\Gamma_{ij}-\textbf{u}_{k}^{n}/\Gamma_{ij}||_{\textbf{B$/\Gamma_{ij}$}}^{2},
\end{equation}
as \textbf{B}$_{i}=R_{i}\textbf{B}R_{i}^{T}$ is a covariance matrix, we get that \textbf{B}/$\Gamma_{ij}=R_{i}\textbf{B}R_{ij}^{T}$ are the restriction of the matrix $B$, respectively, to the sub domain $\Omega_{i}$ and interface $\Gamma_{ij}$ in (\ref{interfacce}) $\textbf{H}_{i}=R_{i}\textbf{H}R_{i}^{T}$, ${\textbf{R}_{i}}=R_{i}\textbf{R}R_{i}^{T}$ the restriction of the matrices $\textbf{H}$, ${\textbf{R}}$ to the sub domain $\Omega_{i}$, $\textbf{u}_{i}^{b}=R_{i}\textbf{u}^{b}$, $\textbf{u}_{i}^{n+1}/\Gamma_{ij}=R_{ij}\textbf{u}_{i}^{n+1}$, $\textbf{u}_{j}^{n}/\Gamma_{ij}=R_{ij}\textbf{u}_{j}^{n}$ the restriction of vectors $\textbf{u}^{b}$, $\textbf{u}_{i}^{n+1}$, $\textbf{u}_{j}^{n}$ to the sub domain $\Omega_{i}$ and interface $\Gamma_{ij}$,  for $i,j=1,2,...,J$. \\
\\
The MPS in \cite{MPS} is used for solving boundary-value problems and as transmission condition on interfaces $\Gamma_{ij}$ for $i,j=1,...,J$ it requires that solution of subproblem on $\Omega_{i}$ at iteration $n+1$ coincides with solution of subproblem on adjacent sub domain $\Omega_{j}$ at iteration $n$; but the 3D-Var DA problem is a variational problem. So, according MPS, we impose the minimization in norm $||\cdot ||_{\textbf{B}/{\Gamma_{ij}}}$ between $\textbf{u}_{i}^{n+1}$ and $\textbf{u}_{j}^{n}$.
The functional $\textbf{J}$ defined in (\ref{3Doper}) as well as all the functionals $\textbf{J}_{i}$ defined in (\ref{funzionaliJi}), are quadratic (hence, convex), so
their unique minimum are obtained as zero of their gradients. In particolar, the functional  $\textbf{J}_{i}$ can be rewritten as follows:
\begin{displaymath}
\begin{split}
\frac{1}{2}(\textbf{w}_{i}^{n+1})^{T}\textbf{w}_{i}^{n+1}+ \frac{1}{2}(\textbf{H}_{i}\textbf{V}_{i}\textbf{w}_{i}^{n+1}-\textbf{d}_{i})^{T}\textbf{R}_{i}^{-1}(\textbf{H}_{i}\textbf{V}_{i}\textbf{w}_{i}^{n+1}-\textbf{d}_{i})+\\ \frac{1}{2}(\textbf{V}_{ij}\textbf{w}_{i}^{n+1}-\textbf{V}_{ij}\textbf{w}_{j}^{n+1})^{T} \cdot (\textbf{V}_{ij}\textbf{w}_{i}^{n+1}-\textbf{V}_{ij}\textbf{w}_{j}^{n}),
\end{split}
\end{displaymath}
where $\textbf{w}_{i}^{n+1}=\textbf{V}_{i}^{T}({\textbf{u}_{i}^{n+1}}-{\textbf{u}_{i}^{b}})$, $\textbf{V}_{i}=R_{i}\textbf{V}R_{i}^{T}$ is the restriction of matrix $\textbf{V}$ to sub domain $\Omega_{i}$, $\textbf{V}_{ij}=R_{i}\textbf{V}R_{ij}^{T}$ is the restriction of matrix $\textbf{V}$ to interfaces $\Gamma_{ij}$, $\textbf{d}_{i}$ the restriction of vector $\textbf{d}$ defined in (\ref{vettored}).\\ The gradients of $\textbf{J}_{i}$ is:
\begin{equation}
\nabla \textbf{J}_{i}(\textbf{w}_{i}^{n+1})=\textbf{w}_{i}^{n+1}+\textbf{V}_{i}^{T}\textbf{H}_{i}^{T}\textbf{R}_{i}^{-1}(\textbf{H}_{i}\textbf{V}_{i}\textbf{w}_{i}^{n+1}-\textbf{d}_{i})+\textbf{V}_{ij}^{T}(\textbf{V}_{ij}\textbf{w}_{i}^{n+1}-\textbf{V}_{ij}\textbf{w}_{j}^{n})
\end{equation}
that can be rewritten as follows
\begin{equation}\label{gradiente}
\nabla \textbf{J}_{i}(\textbf{w}_{i}^{n+1})=(\textbf{V}_{i}^{T}\textbf{H}_{i}^{T}\textbf{R}_{i}^{-1}\textbf{H}_{i}\textbf{V}_{i}+I_{i}+B/\Gamma_{ij})\textbf{w}_{i}^{n+1}-\textbf{c}_{i}+B/\Gamma_{ij} \textbf{w}_{j}^{n},
\end{equation}
where 
\begin{equation}\label{ciMPS}
{c}_{i}=(\textbf{V}_{i}^{T}\textbf{H}_{i}^{T}\textbf{R}_{i}^{-1}\textbf{H}_{i}\textbf{V}_{i}\textbf{d}_{i}),
\end{equation} 
and $I_{i} \in \mathbb{R}^{r_{i} \times r_{i}}$  the identity matrix. \\ From (\ref{gradiente}) by considering the Euler-Lagrange equations we obtain the following systems ${(S_{i}^{MPS})}^{n+1}$:
\begin{equation}\label{DAP}
{(S_{i}^{MPS})}^{n+1}:  \ \ \ A_{i}^{MPS}\textbf{w}_{i}^{n+1}={c}_{i}-\sum_{j \neq i}A_{i,j}\textbf{w}_{j}^{n}, 
\end{equation}
to solve for $n=0,1,...$, where
\begin{equation}\label{AiMPS}
A_{i}^{MPS}=(\textbf{V}_{i}^{T}\textbf{H}_{i}^{T}\textbf{R}_{i}^{-1}\textbf{H}_{i}\textbf{V}_{i}+I_{i}+\textbf{B}/\Gamma_{ij}),
\end{equation}
and $A_{ij} =\textbf{B}/\Gamma_{ij}$, for $i,j=1,...,J$.
\item  For $i=1,...,J$, computation of  $\textbf{u}_{i}^{n+1}$, related to the sub domain $\Omega_{i}$, as follows:
\begin{equation}\label{solDAi}
\textbf{u}_{i}^{n+1}=\textbf{u}_{i}^{b}+\textbf{B}_{i}^{-1} \textbf{V}_{i}\textbf{w}_{i}^{n+1}\\.
\end{equation}
\item Computation of $\textbf{u}^{DA}$, solution of problem 3D-Var DA in (\ref{3Dvar}), obtained by patching together all the vectors $\textbf{u}_{i}^{DA}$, i.e.:
\begin{equation}\label{uda}
\textbf{u}^{DA}(x_{j})=\left\{\begin{array}{ll} \textbf{u}_{i}^{m}(x_{j}) & \textrm{se $x_{j} \in \Omega_{i}$}\\
\textbf{u}_{k}^{m}(x_{j}) & \textrm{se $x_{j} \in \Omega_{k}$ o $x_{j} \in \Omega_{i} \cap \Omega_{k} $}, \end{array}, \right.
\end{equation}
for $i,k=1,...J$, and $m$ corresponding iterations needed to stop of the iterative procedure.
\end{enumerate}

\noindent The functional DD approach in \cite{articolo} provides the minimum of the functional $\textbf{J}$ in (\ref{3Doper}) (defined on the entire domain) as a piecewise function obtained by collecting the minimum of each local functional $\textbf{J}_{i}^{DD-DA}$ defined on sub domain $\Omega_{i}$ and by adding a local constraint about the entire overlap region. It is composed of the following steps:
\begin{enumerate}
\item Decomposition of domain $\Omega$ into a sequence of sub domains $\Omega_{i}$ such that:
\begin{displaymath}
\Omega=\bigcup_{i=1}^{N}\Omega_{i}.
\end{displaymath}
\item Definition of the overlap regions of the sub domains $\Omega_{i}$, as follows:
\begin{displaymath}
\Omega_{ij}=\Omega_{i}\cap \Omega_{j}\ \ \ \textrm{for $i,j=1,...,J$}.
\end{displaymath}
\item Definition of functional restriction ${RO}_{i}$ and functional extension ${EO}_{i}$, as follows.
\begin{definizione}(Functional restriction). 
Let $f$ be a function belonging to the Hilbert space $\mathcal{K}([0,T] \times \Omega)$, that is:
\begin{displaymath}
f(t,x):[0,T] \times \Omega \mapsto \mathbb{R},
\end{displaymath}
then
\begin{displaymath}
RO_{i}:\mathcal{K}([0,T] \times \Omega) \mapsto ([0,T] \times \Omega_{i})
\end{displaymath}
is functional restriction such that 
\begin{displaymath}
RO_{i}(f(t,x)) \equiv f(t,x), \quad (t,x) \in [0,T] \times \Omega_{i},
\end{displaymath}
where $i=1,...,J$.\\
Moreover, for simplicity of notations, we let:
\begin{displaymath}
f_{i}(t,x)\equiv RO_{i}[f(t,x)].
\end{displaymath}
\end{definizione}
\begin{definizione}(Functional extension). 
Let $g_{i}$ con $i=1,...,J$ be a functions belonging to the Hilbert space $\mathcal{K}([0,T] \times \Omega_{i})$, then
\begin{displaymath}
EO_{i}:\mathcal{K}([0,T] \times \Omega_{i}) \mapsto ([0,T] \times \Omega)
\end{displaymath}
is functional extension such that
\begin{displaymath}
EO_{i}(g_{i}(t,x)) =\left\{\begin{array}{ll}
g_{i}(t,x) & x\in \Omega_{i}\\
0 & \textrm{altrimenti}
\end{array} \right.
\end{displaymath}
\end{definizione}
\item For $i=1,...,J$, solution of the $J$ subproblems ${(P_{i}^{DD-DA})}^{n+1}$ for $n=0,1,2,...$,  where
\begin{equation}\label{pDD}
\begin{split}
{(P_{i}^{DD-DA})}^{n+1} \ \ argmin_{\textbf{u}_{i}^{n+1}\in \mathbb{R}^{r_{i}}}\textbf{J}_{i}^{DD-DA}(\textbf{u}_{i}^{n+1})=\\ argmin_{\textbf{u}_{i}^{n+1}\in \mathbb{R}^{r_{i}}}(||\textbf{H}_{i}\textbf{u}^{\textbf{$RO_{i}$}}-\textbf{v}^{\textbf{$RO_{i}$}}||_{\textbf{R}_{i}}^{2}+\\ ||\textbf{u}^{\textbf{$RO_{i}$}}-{({\textbf{u}}^{b})}^{\textbf{$RO_{i}$}}||_{\textbf{B$_{i}$}}^{2}+ \mu ||\textbf{u}^{\textbf{$RO_{i}$}}/\Omega_{ij}-\textbf{u}^{\textbf{$RO_{j}$}}/\Omega_{ij}||_{\textbf{B$_{ij}$}}^{2}),
\end{split}
\end{equation}
where \textbf{B}$_{ij}$ is the restriction of matrix \textbf{B} to overlap region $\Omega_{ij}$, and \textbf{H}$_{i}$, \textbf{B}$_{i}$ the restriction of matrices \textbf{H}, \textbf{B} to sub domain $\Omega_{i}$, according the description in \cite{articolo}.\\
\\
Summarizing, the   DD approach used  in \cite{tesi} addresses the solution of  linear systems arising from the Euler-Lagrange equations  by using MPS method, while the DD approach introduced in \cite{tesi} directly focuses on the functional minimization problem decomposing the least square problem. As the 3D-Var DA is a quadratic functional, we know that its minimization is equivalent to the Euler-Lagrange equation solution. Then, we just need to demonstrate that the local Euler-Lagrange equations give rise to linear systems which are equivalent to those arising by using the MPS approach.  \\
\\
From (\ref{pDD}) we get the following systems ${({S}_{i}^{DD-DA})}^{n+1}$  for $i=1,...,J$ \cite{articolo}:
\begin{equation}\label{system}
{({S}_{i}^{DD-DA})}^{n+1}: \ \ \ {A}_{i}^{DD-DA}\textbf{w}_{i}^{n+1}={c}_{i}^{DD-DA}, 
\end{equation}
to solve for $n=0,1,...$, while the vectors ${c}_{i}^{DD-DA}$ and matrices ${A}_{i}^{DD-DA}$ are defined as follows: 
\begin{equation}\label{ciDD-DA}
{c}_{i}^{DD-DA}=(\textbf{V}_{i}^{T}\textbf{H}_{i}^{T}\textbf{R}_{i}^{-1}\textbf{H}_{i}\textbf{V}_{i}\textbf{d}_{i}),
\end{equation}
\begin{equation}\label{AiDD-DA}
{A}_{i}^{DD-DA}=(\textbf{V}_{i}^{T}\textbf{H}_{i}^{T}\textbf{R}_{i}^{-1}\textbf{H}_{i}\textbf{V}_{i}+I_{i})
\end{equation}
and $I_{i}$,  is the identity matrix.
\item For $i=1,...,J$, computation of  $\textbf{u}_{i}^{n+1}$ on  sub domains $\Omega_{i}$ as in (\ref{solDAi}).
\item Computation of $\textbf{u}^{DA}$,  solution of  3D-Var DA problem  as in (\ref{3Dvar}),  obtained as in (\ref{uda}).
\end{enumerate}
Firstly we note that the following equivalence holds on \cite{tesi}:
\begin{displaymath}
EO_{i}\equiv R_{i}, \quad RO_{i}\equiv R_{i}^{T},
\end{displaymath} 
so that if we let $A \in \mathbb{R}^{NP \times NP}$ an,   for $i=1,...,J$, we consider  $r_{i}$ points of $\Omega_{i}$, it is 
\begin{displaymath}
RO_{i}(A)\equiv R_{i}AR_{i}^{T}.
\end{displaymath} 
\\
\\
Finally we are able to prove the following result.
\begin{proposizione}
Let $\textbf{u}^{DA}$ in (\ref{uda}) be the solution the 3D-Var DA  problem in (\ref{3Dvar}) obtained applying the DD method in \cite{articolo}, that is, by solving for $n=0,1,...$ the linear systems ${(S_{i}^{DD-DA})}^{n+1}$ in (\ref{system}). Similarly, the MPS in \cite{tesi}, provides $\textbf{u}^{DA}$ by solving for $n=0,1,...$, the linear systems ${(S_{i}^{MPS})}^{n+1}$ in (\ref{DAP}). \\ 
We prove that linear systems in (\ref{DAP}) and (\ref{system}) are equivalent.
\end{proposizione}
\begin{proof}
Let us assume that $J=2$ so, we consider the sub domains $\Omega_{1}$, $\Omega_{2}$ and the interfaces $\Gamma_{12}:=\partial \Omega_{1} \cap \Omega_{2}$, $\Gamma_{21}:=\partial \Omega_{2} \cap \Omega_{1}$. \\ By using the DD method in \cite{articolo} it follows that:
\begin{equation}\label{dim1}
{(S_{1}^{DD-DA})}^{n+1}: \ \ \ A_{1}^{DD-DA}\textbf{w}_{1}^{n+1}=c_{1}^{DD-DA} \longrightarrow \textbf{w}_{1}^{n+1}={(A_{1}^{DD-DA})}^{-1}c_{1}^{DD-DA}
\end{equation}
and
\begin{displaymath}
{(S_{2}^{DD-DA})}^{n}:\ \ \ A_{2}^{DD-DA}\textbf{w}_{1}^{n}=c_{1}^{DD-DA} \longrightarrow \textbf{w}_{2}^{n}={(A_{2}^{DD-DA})}^{-1}c_{2}^{DD-DA}
\end{displaymath}
by using the MPS in \cite{tesi}, we get
\begin{equation}\label{dim}
\begin{array}{ll}
{(S_{1}^{MPS})}^{n+1}: \ \ \ A_{1}^{MPS}\textbf{w}_{1}^{n+1}&=c_{1}+A_{12}\textbf{w}_{2}^{n}\\ {(S_{2}^{MPS})}^{n+1}: \ \ \ A_{2}^{MPS}\textbf{w}_{2}^{n+1}&=c_{2}+A_{21}\textbf{w}_{1}^{n}.
\end{array}
\end{equation}
We prove the equivalence between linear systems ${(S_{1}^{DD-DA})}^{n+1}$ in (\ref{dim1}) and ${(S_{1}^{MPS})}^{n+1}$ in (\ref{dim}), i.e. we prove that solutions obtained from ${(S_{1}^{DD-DA})}^{n+1}$ and ${(S_{2}^{DD-DA})}^{n}$ in (\ref{dim1}) satisfy ${(S_{1}^{MPS})}^{n+1}$ in (\ref{dim}). \\
\\
Replacing $\textbf{w}_{1}^{n+1}$ by ${(A_{1}^{DD-DA})}^{-1}c_{1}^{DD-DA}$ in (\ref{dim}), it follows that:
\begin{equation}\label{dim2}
A_{1}^{MPS}(A_{1}^{DD-DA})^{-1}c_{1}^{DD-DDA}=c_{1}+A_{12}\textbf{w}_{2}^{n};
\end{equation}
the matrix $A_{1}^{MPS}$ in (\ref{AiMPS}) can be rewritten as 
\begin{equation}
A_{1}^{MPS}:=(\textbf{V}_{1}^{T}\textbf{H}_{1}^{T}\textbf{R}_{1}^{-1}\textbf{H}_{1}\textbf{V}_{1}+I_{1}+\textbf{B}/\Gamma_{12})=A_{1}^{DD-DA}+A_{12},
\end{equation}
where $A_{1}^{DD-DA}$ is defined in (\ref{AiDD-DA}).
Then, the (\ref{dim2}) becomes:
\begin{equation}\label{dim3}
(A_{1}^{DD-DA}+A_{12})(A_{1}^{DD-DA})^{-1}c_{1}^{DD-DA}=c_{1}+A_{12}\textbf{w}_{2}^{n}.
\end{equation} 
Replacing $\textbf{w}_{2}^{n}$ with ${(A_{2}^{DD-DA})}^{-1}c_{2}^{DD-DA}$ in (\ref{dim3}), it follows that:
\begin{displaymath}
c_{1}^{DD-DA}+A_{12}((A_{1}^{DD-DA})^{-1}c_{1}^{DD-DA}){|}_{\Gamma_{12}}=c_{1}+A_{12}((A_{2}^{DD-DA})^{-1}c_{2}^{DD-DA}){|}_{\Gamma_{12}}
\end{displaymath}
then, we obtain that
\begin{equation}\label{ug}
c_{1}^{DD-DA}=c_{1}.
\end{equation}
From (\ref{ciMPS}) and (\ref{ciDD-DA}) we have that   
\begin{displaymath}
c_{1}^{DD-DA}=(\textbf{V}_{i}^{T}\textbf{H}_{i}^{T}\textbf{R}_{i}^{-1}\textbf{H}_{i}\textbf{V}_{i}\textbf{d}_{i})=c_{1},
\end{displaymath}
so the proof is complete.
\end{proof}
\section{Conclusions}

In this note, we reviewed two DD approches, namely the first one is  the functional DD appiled to the variational model and the latter one is the discrete Multiplicative Parallel Schwarz applied to linear system arising from the Euler-Lagrange equations arising from the 3D-VarDA least square problem.  We proved equivalence between these approaches, according to the fact that the  VarDA model is quadratic.

\end{document}